\documentclass{amsart}

\newtheorem{Theorem}{Theorem}[section]
\newtheorem{Lemma}[Theorem]{Lemma}
\newtheorem{Corollary}[Theorem]{Corollary}

\usepackage{pst-node}
\usepackage[dvips]{graphicx}    % grafika w zwyk³ym

\begin{document}

\date{ \today}

\title[]{Graphical cyclic permutation groups}
\author{Mariusz Grech}
\address{Institute of Mathematics, University of Wroclaw \\
pl.Grunwaldzki 2, 50-384 Wroclaw, Poland}
\email{Mariusz.Grech@math.uni.wroc.pl}

%\keywords{Graph, automorphism group, permutation group, direct product.}

\newcommand{\cnd}  {$\Box$}
\newcommand{\V}     {\mbox{$\cal V$}}
\newcommand{\AUT}     {\mbox{$\it AUT$}}
\newcommand{\B}     {\mbox{$ {\cal A}_2 $}}
\newcommand{\A}     {\mbox{${\cal A}_1$}}
\newcommand{\G}     {\Gamma}
\newcommand{\<}     {\langle }
\renewcommand{\>}   {\rangle }
\newcommand{\s}     {\mbox{$\bf\sigma$}}
\newcommand{\AB}     {\mbox{$\times$}}
\newcommand{\AC}     {\mbox{$\otimes$}}
\newcommand{\row}[2]{{#1}^{(#2)}}
\newcommand{\Au}    {\rm Aut}
\renewcommand{\wr}   {\hspace{.5mm} wr \hspace{.5mm}}
\newcommand{\teki}{\textbullet \hspace{2mm}}

\begin{abstract}
We  establish conditions for a permutation group generated by a single permutation of a prime power order to be an automorphism group of a graph or an edge-colored graph. This corrects and generalizes the results of the two papers on cyclic permutation groups published in 1978 and 1981 by S. P. Mohanty, M. R. Sridharan, and S. K. Shukla.
\end{abstract}

\maketitle

\section{Introduction}
It is well known that while every abstract group is isomorphic to the automorphism group of a graph, not every permutation group can be represented directly as the automorphism group of a concrete graph.

The problem of representatibility of a permutation group $A=(A,V)$ as the full automorphism group of a graph $G=(V,E)$ was studied first for regular permutation groups. In particular, the question which  abstract groups have a regular representation as an automorphism group of a graph (the so-called $GRR$) has received considerable attention. In the language of permutation groups this is the same as to ask which regular permutation groups are automorphism groups of the graphs. There were many partial results in this area (see for
instance \cite{imr1,imr2,imrwa,no,nowa1,nowa2,wa1,wa2,wa3}). The full
characterization has been obtained by Godsil \cite{god} in 1979.  
In \cite{ba1}, L. Babai uses the result of Godsil to prove a similar characterization in the case of directed graphs.

In \cite{mss1, mss2}, S. P. Mohanty, M. R. Sridharan, and S. K. Shukla,
consider cyclic permutation groups (i.e., generated by a single permutation) whose order is $p^n$ for a prime $p$. In \cite[Theorem~3]{mss2}, they describe all such groups for $p>5$ that are automorphism groups of  graphs. However, although the result is true, the proof contains an essential gap. Moreover, the authors make a false claim that there are no such groups for $p=3$ or $p=5$. 
We correct these results here and generalize  as follows. First, we consider also the cases $p\in \{ 2,3,5\}$, which are different, and completely ignored in \cite{mss1, mss2}. 
In fact the case $p>5$ is the simplest and the least interesting. 
The difficulty and beauty of the problem is hidden in these three cases.
At second, we generalize the results describing the cyclic permutation groups of a prime power order that are automorphism groups of \emph{edge-colored graphs}. As in many earlier cases (see e.g. results in \cite{pei,grekis1,gre1,gre2}) edge-colored graphs turn out to be a more appropriate setting for these kind of problems, giving simpler, and more elegant formulations of results. 

In fact, the general problem has been considered already by H.~Wielandt in \cite{wie}.
Permutation groups that are automorphism groups of colored graphs were called $2^*$-closed. 
In \cite{kis1}, A. Kisielewicz introduced the so-called \emph{graphical complexity of permutation groups}. By $GR(k)$ we denote the class of automorphism groups of $k$-\emph{colored graphs}, by which we mean the complete graphs whose edges are colored with at most $k$ colors. By $GR$ we denote the union of all classes $GR(k)$ (which is the class of all $2^*$-closed groups).
Moreover, we put $GR^*(k)= GR(k) \setminus GR(k-1)$, and for a permutation group $A$, we say that $A$ has a \emph{graphical complexity} $k$ if $A \in GR^*(k)$. 
Then, $GR(2)$ is the class of automorphism groups of simple graphs.

Now, the main natural problem in this subject is the characterization of
permutation groups in $GR$, i.e.,  those permutation groups that are
automorphism groups of $k$-colored graphs for any $k$. The most exciting question stated in \cite{kis1} is however whether the hierarchy $GR(k)$ is a real hierarchy, that is, whether there are permutation groups in each of the classes $GR^*(k)$. For now we know only that $GR^*(k)$ are nonempty for all $k\leq 6$, which means in particular that there is a 6-colored graph whose automorphism group is different from all the automorphism groups of $k$-colored graphs with $k \leq 5$; such a graph on $n=32$ vertices was constructed in \cite{grekis2}.

There are various approaches to the topic. In \cite{grekis1} it is shown that
the direct sum of two groups belonging to $GR(k)$ belongs itself to $GR(k+1)$,
and most often it belongs to $GR(k)$. Similar results are obtained for other products. In \cite{grekis2}, we have described almost all $k$-colored graphs with the highest degree of symmetry. In particular, we have shown that if a $k$-colored graph is edge-transitive and color-symmetric, then $k \leq 5$. 

 In this paper we  determine when a cyclic permutation group $A$ of a prime power order belongs to $GR$, and show that if $A \in GR$, then $A \in GR(3)$. We also characterize, when it is in $GR(2)$, and when in $GR^*(3)$. In section~\ref{definicje}, we recall some definitions concerning edge-colored graphs and permutation groups. We also recall two results from \cite{grekis1}, and prove their generalizations we need in the sequel. In section~\ref{pierwsze}, we complete the consideration of cyclic permutation groups of prime order started in \cite{mss1, mss2}. In particular, we show that there are cyclic permutation groups of order $3$ and $5$ which belong to $GR(2)$, contradicting the claim given in \cite{mss1}. In section~\ref{potegi}, we give a correct proof of the theorem stated in \cite{mss2} concerning cyclic permutation group of prime power order $p^n$ with $p>5$, and complete the research considering the cases for $p= 2,3,5$.

\section{Definitions and basic facts} \label{definicje}
We assume that the reader has the basic knowledge in the areas of graphs and permutation groups.
Our terminology is standard and for nondefined notions the reader is referred to \cite{ba,yap}.

A {\it $k$-colored graph} (or more precisely $k$-\emph{edge-colored graph})
is a pair $G = (V,E)$, where $V$ is the set of vertices,
and $E$ is an \emph{edge-color function}
from the set $P_2(V)$ of two elements subsets of $V$ into
the set of colors $\{ 0, \ldots, k-1\}$ (in other words, $G$ is a complete simple
graph with each edge colored by one of $k$ colors).
In some situations it is helpful to treat the edges colored $0$ as missing.
In particular, a $2$-colored graph can be treated as a usual simple graph.
Generally, if no confusion can arise, we omit the adjective ,,colored''.
By a (sub)graph of $G$ \emph{spanned} by a subset $W \subseteq V$
we mean $G'=(W,E')$ with $E'(\{v,w\})=E(\{v,w\})$, for all $v,w \in W$.

Let $v,w \in V$ and $i \in \{ 0, \ldots, k-1 \}$.
If $E(\{ v,w\} )=i$, then we say that $v$ and $w$ are $i$-{\it neighbors}.
Moreover, for a set $X \subseteq \{ 0, \ldots, k-1 \}$, we say that a vertex $w$ is an $X$-neighbor of a vertex $v$ if
there is a color $i \in X$ such that $w$ is an $i$-neighbor of $v$.
By $d_i(v)$ ($i$-{\it degree} of a vertex $v$) we denote the number of $i$-neighbors of $v$.
For $X \subseteq \{0, \ldots, k-1\}$, we say that $G$ is $X$-{\it connected}, if
for every $v,w \in V$ there is a path $v=v_0, v_1, \ldots, v_n=w$ in $G$ such
that the color of each edge $\{ v_i, v_{i+1} \}$ belongs to $X$.
Obviously, for a $k$-colored graph $G=(V,E)$,
and for the sets $X,Y \subseteq \{ 0, \ldots, k-1 \}$ such that $X \cup Y = \{ 0, \ldots, k-1 \}$,
$G$ is either $X$-connected or $Y$-connected. In particular, there is always a
color $p$ such that $G$ is $(\{ 0,\ldots, k-1 \} \setminus \{ p \})$-connected.

An automorphism of a colored graph $G$ is a permutation $\sigma$ of the set $V$ of vertices
preserving the edge function: $E(\{ v, w\}) = E( \{ \sigma(v), \sigma(w) \})$, for all $v,w \in V$. The group of automorphisms of $G$ will be denoted by
$Aut(G)$, and considered as a permutation group $(Aut(G),V)$ acting on the set of the vertices $V$.

Permutation groups are treated up to permutation isomorphism. 
Generally, a permutation group $A$ acting on a set $V$ is denoted $(A,V)$ or just $A$, if the set $V$ is clear or not important. By $S_n$, we denote the symmetric group on $n$ elements, and by $I_n$ the trivial one element group acting on $n$ elements (consisting of the identity only). 
By $C_n$ we denote a regular action of ${\mathbb Z}_n$. 
In particular, $S_2 = C_2$. 
Finally, $D_n$ denotes the dihedral group of symmetries of $n$-cycle i.e., the group of automorphisms of a graph $G=(V,E)$ with  $V = \{ v_0, \ldots, v_{n-1}\}$,
$E(\{v_i,v_{(i+1\; {\rm mod}\, n)})=1$ for all $i$, and $E(v_i,v_j)=0$, otherwise.
It is clear that $C_n$ is a subgroup of $D_n$ of index two.

Later, we will use two kinds of products of permutation groups:
{\it Direct sum}. For permutation groups $(A,V), (B,W)$,  by the direct sum of $A$ and $B$ we mean the permutation group $(A \oplus B, V \cup W)$ with
the action given by
\begin{eqnarray}
\nonumber
(a,b)(x)&=& \left\{\begin{array}{cl} a(x) & \textrm{for } x \in V, \\
b(x) & \textrm{for } x \in W.
\end{array}\right.
\end{eqnarray}

{\it Parallel product}. For the permutation group $(A,V)$, the
parallel product $\row{A}{n}$ is the permutation group $(A,V \times \{1, \ldots, n\})$ with the following natural action.
$$a((v_1,k))=(a(v_1), k).$$

Now, we recall two theorems which are proved in \cite{grekis1} and will be used later.

\begin{Theorem} \cite[Corollary 3.5]{grekis1}\label{suma1}
Let $A = A_1 \oplus A_2$ be a direct sum. Then, $A\in GR$ if and only if each
of $A_1$ and $A_2$ belongs to $GR$ or $A$ is equal to $I_2 = I_1 \oplus I_1$.
\end{Theorem}

\begin{Theorem} \cite[Lemma 3.1 and Theorem 4.1]{grekis1} \label{suma2}
Let $A_1, A_2 \in GR(k)$, for some $k \geq 2$. Then,
\begin{enumerate}
    \item $A_1 \oplus A_2 \in GR(k+1)$.
    \item If $A_1 \ne A_2$, then $A_1 \oplus A_2 \in GR(k) $. 
\item $A_1 \oplus I_n \in GR(k) \cup \{I_2\}$.  
\end{enumerate}
\end{Theorem}

Recall also (see \cite{grekis1}, for instance) that $I_1 \in GR(0)$, $I_2
\notin GR$, $I_n \in GR^*(3)$, for $n \in \{ 3,4,5\}$, and $I_n \in GR(2)$, otherwise. This describes completely the case of permutation groups of order one (which, as we see, is not quite trivial). In further consideration, we assume that the
order of a cyclic permutation group is at least two.
Later on, we will need one more lemma.

\begin{Lemma}\label{trywialne}
Let $k \geq 1$ and $B \notin GR(k)$ be a permutation group such that for every $k$-colored graph $G$, with the property $B \subseteq Aut(G)$, there is a permutation $f \in Aut(G)\setminus B$  preserving all the orbits of $B$. Then, $B \oplus C \notin GR(k)$ for every permutation group $C$.
\end{Lemma}

\begin{proof}
Let $B=(B,V)$ and $G'$ be a $k$-colored graph such that $B \oplus C \subseteq Aut(G')$.
Then, obviously, the graph $G$, spanned by the set $V$, has the mentioned property.
Let $f \in Aut(G) \setminus B$ be a permutation which preserves all the orbits of $B$.
By $f'$ we denote a permutation which acts as $f$ on $V$ and fixes all other vertices of $G'$.
Obviously, $f' \notin B \oplus C$. We show that $f' \in Aut(G')$.

We have to show that the colors of the edges of the graph $G'$ are preserved by $f'$.
If an edge $e$ is contained in the graph $G$, then $f'(e)=f(e)$ and $E(f'(e))=E(f(e))=E(e)$ as required.
If neither of the ends of $e$ belongs to $V$, then $f'(e)=e$ and the statement is still true.
The only nontrivial case is the edges of the form $e=\{ v,w\}$, where $v \in V$ and $w \notin V$.
Then, $f'(e)=\{ f(v),w\}$. Since $f$ preserves all orbits of $B$, there exists $b \in B$ such that $f(v) = b(v)$.
Consequently, $f'(e) = b'(e)$, where $b'=(b,1) \in B \oplus C$. Hence $E(f'(e))=E(b'(e))=E(e)$, as required.
This shows that $f' \in Aut(G')$ and completes the proof of lemma.
\end{proof}

\section{Primes} \label{pierwsze}
In this section we give a complete description of graphical complexity of cyclic permutation groups of prime order, started in \cite{mss1,mss2}. 
These groups have the form $\row{C_p}{r} \oplus I_q$. 
The result from \cite{mss1} can be rewritten as follows.

\begin{Theorem}\cite[Theorem 3]{mss1} \label{rz2}
Every permutation group of order two belongs to $GR(2)$.
\end{Theorem}

\begin{Theorem}\cite[Theorem 2]{mss1}\label{prime>5}
Let $p>5$ be a prime. 
Then, $\row{C_p}{r} \in GR(2)$ if and only if $r \geq 2$.
\end{Theorem}

In \cite{mss2}, the latter is a corollary of \cite[Theorem~3]{mss2}, which is given with an incorrect proof. The direct correct proof of Theorem~\ref{prime>5} can be found in \cite{mss1}.

We consider two remaining primes $3$ and $5$. We start with simple observations.
It is an easy and well known fact (see for instance \cite{chao} and \cite{kag}) that the groups $C_n$, for $n >2$, do not belong to $GR(2)$. 
The same is true for any class $GR(k)$, where $k \geq 2$. The reason is that    
for every $k$-colored graph $G$, if $Aut(G)$ contains $C_n$, then $Aut(G)$ contains $D_n$. 
Hence, $C_n \notin GR$. It follows, by Theorem~\ref{suma1}, that $C_n \oplus I_m \notin GR$ for any $n\geq 3$ and $m\geq 1$. In contrast, we have

\begin{Lemma}\label{p3}
$\row{C_n}{2} \in GR(3)$ for every $n \geq 3$. 
\end{Lemma}

\begin{proof} 
We construct a 3-colored graph $G=(V,E)$. 
Let $V = O_1 \cup O_2$, where the sets $O_1=\{v_0, \ldots, v_{n-1}\}$ and $O_2=\{w_0, \ldots, w_{n-1}\}$ are the orbits of $\row{C_n}{2}$. 
\begin{eqnarray}
\nonumber E(\{v,w\})&=&\left\{
\begin{array}{cl}
1 & \textrm{for } \{v,w\} = \{v_i,v_{i+1}\}, \\ 
1 & \textrm{for } \{v,w\} = \{v_i,w_i\}, \\ 
2 &  \textrm{for } \{v,w\} = \{v_i,w_{i+1}\}, \\
0 & \textrm{otherwise}.
\end{array}\right.
\end{eqnarray} 

(Throughout the paper, the addition in the indices is modulo $n$; thus $(n-1)+1=0$. In Figure~\ref{nietr} an example for $n=3$ is presented.)

\begin{figure}%%%%%%%%%%%%%%%%%%%%%%%%%%%%%%%%%%%%%%%%%%%%%%%%%%%%%%%%%%%%%%%%%%%%%
    \psset{linewidth=0.03, xunit=2cm, yunit=1.5cm}
\begin{pspicture}(6,7)(0,0)

\cnode(3,1){0.1}{b1}
\cnode(2,2.8){0.1}{b2}

\cnode(4,2.8){0.1}{b3}

\cnode(1,1){0.1}{c1}
\cnode(5,1){0.1}{c2}
\cnode(3,4.6){0.1}{c3}

\psset{linewidth=0.05}
\ncline[linecolor=red, linestyle=dashed]{-}{b2}{c1}
\ncline[linecolor=red, linestyle=dashed]{-}{b3}{c3}
\ncline[linecolor=red, linestyle=dashed]{-}{b1}{c2}

\psset{linewidth=0.05}
\ncline[linecolor=black]{-}{b2}{b1}
\ncline[linecolor=black]{-}{b2}{b3}
\ncline[linecolor=black]{-}{b3}{b1}

\ncline[linecolor=black]{-}{c1}{b1}
\ncline[linecolor=black]{-}{b3}{c2}
\ncline[linecolor=black]{-}{b2}{c3}

\psset{linewidth=0.04}
\cnode[linecolor=black](1,0.4){0}{qq}
\cnode[linecolor=black](1.7,0.4){0}{ww}
\ncline[linecolor=black]{-}{qq}{ww}

\rput(2.1,0.44){color 1}
\psset{linewidth=0.05}
\cnode[linecolor=red, linestyle=dashed](1,0.1){0}{qq2}
\cnode[linecolor=red, linestyle=dashed](1.7,0.1){0}{ww2}

\ncline[linecolor=red, linestyle=dashed]{-}{qq2}{ww2}

\rput(2.1,0.14){color 2}

\end{pspicture}
\caption{A $3$-edge-colored graph $G$ such that $Aut(G) = \row{C_3}{2}$.}\label{nietr}
\end{figure}
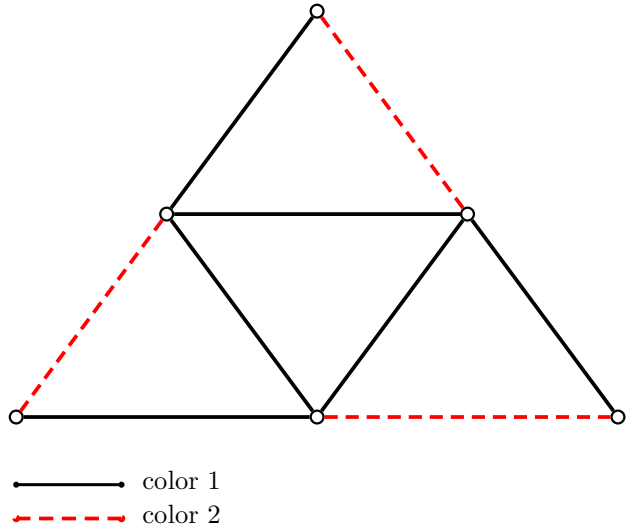

It is clear that $Aut(G) \supseteq \row{C_n}{2}$. 
We prove the opposite inclusion. 
A vertex that belongs to $O_1$ has $1$-degree equal to $3$. 
A vertex that belongs to $O_2$ has $1$-degree equal to $1$. 
Therefore, $Aut(G)$ preserves the sets $O_1$ and $O_2$. 
Since, $Aut(G)$ is transitive on $O_1$, and each member of $O_1$ has exactly one $1$-neighbor that belongs to $O_2$, the action of $Aut(G)$ on $O_1$ and $O_2$ is parallel. Since a subgraph of $G$ spanned by $O_1$ is a $|O_1|$-cycle, $Aut(G)$ is either $\row{C_n}{2}$ or $\row{D_n}{2}$. 
We have to exclude at least one member of $\row{D_n}{2} \setminus \row{C_n}{2}$.
We choose a permutation 
$$\s = (v_1, v_{n-1})(v_2,v_{n-2}) \cdots (v_{\lfloor\frac{n-1}{2}\rfloor},v_{\lceil\frac{n+1}{2}\rceil}) 
(w_1, w_{n-1}) \cdots (w_{\lfloor\frac{n-1}{2}\rfloor},w_{\lceil\frac{n+1}{2}\rceil}).
$$

The permutation $\s$ fixes the vertex $v_0$. 
The vertex $w_1$ is $2$-neighbor of $v_0$ and $w_{n-1}$ is not. Therefore, $\s \notin Aut(G)$, and consequently, $Aut(G) =  \row{C_n}{2}$.
\end{proof}

By Theorem~\ref{suma2}, we have 
\begin{Corollary}
If $q \geq 0$, then $\row{C_n}{2} \oplus I_m \in GR(3).$
\end{Corollary}

As it was observed in \cite{mss2}, $\row{C_p}{2} \notin GR(2)$ for $p = \{3,5\}$. 
We give a little more detailed proof of this fact showing at the same time that $\row{C_4}{2} \notin GR(2)$. 
This will be also helpful in the next section where we give the similar, but less detailed, proof of a similar lemma.

\begin{Lemma}\label{nie3}
$\row{C_n}{2} \notin GR(2)$ for $n \in \{3,4,5\}$. 
\end{Lemma}

\begin{proof}
Let $O_1 = \{v_0, \ldots, v_{n-1}\}$, $O_2 = \{w_0, \ldots, w_{n-1}\}$ be the orbits of $\row{C_n}{2}$. 
We consider an action of $\row{C_n}{2}$ on the set of the edges. 
We have three types of orbits under this action. 
An orbit of the first type consists of the edges of the form $\{v_i,v_j\}$. 
An orbit of the second type consists of the edges of the form $\{w_i,w_j\}$. 
It is obvious that any coloring of these two types of orbits is preserved by the group $D_n \oplus D_n$. 
Finally, we have $n$ orbits consisting the edges of the form $\{v_i,w_j\}$. 
These orbits are represented by the edges $\{v_0,w_j\}$, $j \in \{0, \ldots, n-1\}$. 

We say that two colorings of the orbits of the third type are similar if one coloring can be obtained from the other by permutation numbers of colors and cyclic permutation the elements of $O_2$. 
Then, we have two non-similar colorings of the orbits of the third type for $n=3$, either all orbits in the same color or the orbit represented by the edge $\{v_0,w_0\}$ in one color and two other orbits in another color. 
For $n \in \{4,5\}$, we have four non-similar colorings, represented by the following. 
\begin{itemize}
    \item[1.] All orbits in the same color.
\item[2.] An orbit represented by the edge $\{v_0, w_0\}$ in one color and the rest of the orbits in another color. 
\item[3.] The orbits represented by the edges $\{v_0,w_2\}$, and $\{v_0,w_3\}$ in one color and the rest of the orbits in another color. 
\item[4.] The orbits represented by the edges $\{v_0,w_1\}$, and $\{v_0,w_{n-1}\}$ in one color and the rest of the orbits in another color.
\end{itemize}  

It is easy to see that, with one exception, all these colorings are preserved by permutation 
$$\s = (v_1,v_2)(w_1,w_2) \textrm{ for } n=3,$$ 
$$\s = (v_1,v_3)(w_1,w_3) \textrm{ for } n=4,$$  
and  $$\s = (v_1,v_4)(v_2,v_3)(w_1,w_4)(w_2,w_3) \textrm{ for } n=5.$$
The exceptional case is the third coloring for $n=4$. 
However, in this case, the coloring is preserved by permutation 
$$\tau = (v_0,v_1)(v_2,v_3)(w_0,w_1)(w_2,w_3).$$
Obviously, $\{\s,\tau\} \subset (D_n \oplus D_n) \setminus \row{C_n}{2}$. 
Therefore, $\row{C_n}{2} \notin GR(2)$. 
\end{proof}

By Lemma~\ref{trywialne}, we have the following. 
\begin{Corollary}
$\row{C_n}{2} \oplus I_m \notin GR(2)$ for $n \in \{3,4,5\}$ and $m \geq 0$. 
\end{Corollary}

In \cite{mss1}, it is noted (without a proof) that there exists no permutation group of prime order $p \in \{3,5\}$ which belongs to $GR(2)$. 
However, for a larger number of nontrivial orbits we have the following.

\begin{Lemma}\label{p2}
Let $r>2$. Then, $\row{C_n}{r} \in GR(2)$ for every $n \geq 3$. 
\end{Lemma}

\begin{proof} 
We construct a graph $G=(V,E)$.
Let $V = \bigcup_{i=1}^r O_i$, where the sets $O_i = \{ v_0^i, \ldots, v_{n-1}^i\}$, $i \in \{1, \ldots r\}$ are the orbits of the group $\row{C_n}{r}$.
$$E(\{v,w\})=\left\{
\begin{array}{cl}
1 & \textrm{for } \{v,w\} = \{v_i^1, v_{i+1}^1\}, \\ 
1 & \textrm{for } \{v,w\} = \{v_i^j, v_i^l\},  \textrm{ where either } l = j+1 \textrm{ or } j=1, l=3,\\ 
1 &  \textrm{for } \{v,w\} = \{v_i^1,v_{i+1}^2\}, \\
0 & \textrm{otherwise}.
\end{array}\right.$$

(See Figure~\ref{c33} for $n=3$ and $r=3$.)

\begin{figure}%%%%%%%%%%%%%%%%%%%%%%%%%%%%%%%%%%%%%%%%%%%%%%%%%%%%%%%%%%%%%%%%%%%%%
    \psset{linewidth=0.03, xunit=2cm, yunit=1.5cm}
\begin{pspicture}(6,3.8)(0,0)

\cnode(3,0){0.1}{b1}
\cnode(2,1.8){0.1}{b2}

\cnode(4,1.8){0.1}{b3}

\cnode(1,0){0.1}{c1}
\cnode(5,0){0.1}{c2}
\cnode(3,3.6){0.1}{c3}

\cnode(4,0.6){0.1}{d1}
\cnode(2,0.6){0.1}{d2}
\cnode(3,2.4){0.1}{d3}

 \psset{linewidth=0.05}

\ncline[linecolor=black]{-}{b2}{c1}
\ncline[linecolor=black]{-}{b3}{c3}
\ncline[linecolor=black]{-}{b1}{c2}

\ncline[linecolor=black]{-}{b2}{b1}
\ncline[linecolor=black]{-}{b2}{b3}
\ncline[linecolor=black]{-}{b3}{b1}

\ncline[linecolor=black]{-}{c1}{b1}
\ncline[linecolor=black]{-}{b3}{c2}
\ncline[linecolor=black]{-}{b2}{c3}

\ncline[linecolor=black]{-}{d2}{b1}
\ncline[linecolor=black]{-}{d1}{b3}
\ncline[linecolor=black]{-}{d3}{b2}

\ncline[linecolor=black]{-}{c1}{d2}
\ncline[linecolor=black]{-}{d1}{c2}
\ncline[linecolor=black]{-}{d3}{c3}

\end{pspicture}
\caption{A $2$-egde-colored graph $G$ such that $Aut(G) = \row{C_3}{3}$.}\label{c33}
\end{figure}

It is clear that $Aut(G) \supseteq \row{C_n}{r}$. 
We prove the opposite inclusion. 
A vertex that belongs to $O_1$ has $1$-degree equal to $5$. 
A vertex that belongs to $O_2$ has $1$-degree equal to $3$. 
The rest of the vertices have $1$-degree at most $3$. 
Moreover, if $r \geq 4$, then a vertex that belongs to $O_r$ has $1$-degree equal to $1$, and if $r = 3$, then a vertex that belongs to $O_3$ has $1$-degree equal to $2$. 
Therefore, $Aut(G)$ preserves the sets $O_1$ and $O_r$. 
In addition, for $k >2$, a vertex that belongs to $O_k$ has exactly one $1$-neighbor that does not belong to $O_{k+1}$. This $1$-neighbor belongs to $O_{k-1}$. This proves, inductively, that $Aut(G)$ preserves all the sets $O_i$, $i \in \{1, \ldots, r\}$. 
Moreover, $Aut(G)$ acts parallelly on the sets $O_i$, $i \in \{2, \ldots, r\}$. 
Since a vertex that belongs to $O_3$ has exactly one $1$-neighbor that belongs to $O_1$, $Aut(G)$ acts parallelly on all the sets $O_i$, $i \in \{1, \ldots, r\}$. 
Since, a subgraph of $G$ spanned by $O_1$ is a $|O_1|$-cycle, $Aut(G)$ is either $\row{C_n}{r}$ or $\row{D_n}{r}$. 
It is enough to exclude one element  $\s \in \row{D_n}{r} \setminus \row{C_r}{2}$.
We choose a permutation 
$$\s = (v_1^1, v_{n-1}^1)(v_2^1,v_{n-2}^1) \cdots (v_{\lfloor\frac{n-1}{2}\rfloor}^1,v_{\lceil\frac{n+1}{2}\rceil}^1) \cdots
(v_1^r, v_{n-1}^r) \cdots (v_{\lfloor\frac{n-1}{2}\rfloor}^r,v_{\lceil\frac{n+1}{2}\rceil}^r).
$$
But $v_{n-1}^1$ is a $1$-neighbor of $v_0^2$ and $v_1^1$ is not. Moreover, $\s$ fixes $v_0^2$. 
Hence, $\s \notin Aut(G)$. 
This completes the proof of the lemma. 
\end{proof}

By Theorem~\ref{suma2}, we have
\begin{Corollary}
Let $r>2$ and $m \geq 0$. Then, $\row{C_n}{r} \oplus I_m \in GR(2).$
\end{Corollary}

Summarizing this section, we have the following description of the graphical complexity of cyclic permutation groups of prime order. 

\begin{Theorem} \label{pier}
Let $p$ be a prime, $r\geq 1$, $q\geq 0$ and $A = \row{C_p}{r} \oplus I_q$. Then, 
\begin{enumerate}
    \item $A \notin GR$, for $r=1$ and $p\ne 2$, 
\item $A \in GR^*(3)$, for $r=2$ and $p \in \{ 3,5 \}$, 
\item $A \in GR(2)$, otherwise. 
\end{enumerate}
\end{Theorem}

\section{Prime powers} \label{potegi}
In this section, 
we assume that the order of the cyclic permutation group $A$ is a prime power $p^n$ for some $n>1$. 
It is obvious that every orbit of $A$ has cardinality $p^i$, where $0 \leq i \leq n$. 
Moreover, there is at least one orbit of cardinality~$p^n$. 
In this section, we will always assume that a generator $a$ of $A$ acts on each orbit $O=\{v_1, \ldots, v_{p^i-1}\}$ in the following way. 
$$a(v_j) = a(v_{j+1}).$$
(Here, the addition in the indices is modulo $n = p^i$). 
The result \cite[Theorem~3]{mss2} can be rewritten in the following form.

\begin{Theorem}\cite[Theorem 3]{mss2} \label{poteg}
Let $A$ be a cyclic permutation group of order $p^n$, where $p>5$ is prime. 
Then, $A \in GR(2)$ if and only if $A$ has at least two nontrivial orbits.
\end{Theorem}

Unfortunately, this theorem is not really proved in \cite{mss2}. 
The proof is based on \cite[Theorem~2]{mss2}, which in turn has an incorrect proof.
Examining carefully the proof one can find that what has been really proved is the following.

\begin{Theorem}\label{s³absza-wersja}
Let $A$ be a cyclic permutation group, $m_1, m_2, \ldots, m_r$ be the cardinalities of the orbits of $A$.
Let $m_1>5$ and $m_1$ divides $m_i$ for all $i \in \{1, \ldots, r\}$. 
If for every $i \ne j$, such that $\{i,j\} \subseteq \{1, \ldots, r\}$, either $m_i = m_j$ or $\gcd(m_i,m_j)=m_1$, then $A \in GR(2)$. 
\end{Theorem}

Theorem~\ref{s³absza-wersja} does not imply Theorem~\ref{poteg}. 
We complete now the proof of Theorem~\ref{poteg}

\vspace{2mm}

\noindent 
{\bf The proof of Theorem~$\ref{poteg}$}
By Theorem~\ref{suma2}, we may assume that $A$ has no nontrivial orbits. 
The part of the proof, in which $A$ has the only one nontrivial orbit is as in \cite{mss2}. 

Let  
$O_j=\{ v^j_0, \ldots, v^j_{p^{i_j}-1} \}$, where $j \in \{1, \ldots, r\}$, $i_j \leq n$ and $r$ is at least $2$, be the orbits of $A$.
We assume that $i_l \geq i_{l+1}$, for all $l$. It is clear that $i_1=n$.

We construct a graph $G=(V,E)$ such that $Aut(G)=A$. 
We define $V = \bigcup_{l=1}^r O_l$,  
$$ E(\{ v,w \}) = \left\{ 
\begin{array}{cl} 
1 & \hspace{-2mm} \textrm{for }  \{v,w\} = \{v_h^1,v_{h+1}^1\},  \\
1 & \hspace{-2mm}\textrm{for }  \{v,w\} = \{v_j^1,v_{j+e}^2\}, e \in \{ 0,1,3\}, \\ 
1 & \hspace{-2mm} \textrm{for } \{v,w\}=\{v_h^l,v_{h+1}^{l+1}\}, \\
0 & \hspace{-2mm} \textrm{otherwise}.
\end{array} \right. $$ 
(Again, the addition in the indices is modulo $n = p^n, p^{i_2}$, and $p^{i_{l+1}}$, respectively.)

It is clear that $Aut(G) \supseteq A$. 
We have to show the opposite inclusion. 
Since $A$ acts regularly on $O_1$, it is enough to show that $Aut(G)$ stabilizes the set $O_1$ and the stabilizer of $v_0^1$ is trivial. 
Note that $O_1$ consists of all the vertices that have $1$-degree equal to $5$. 
Therefore, $Aut(G)$ preserves the set $O_1$. 
Observe that the set $O_i$, $i\geq 2$, consists of all the $1$-neighbors of the members of $O_{i-1}$ that do not belong to $\bigcup_{j<i} O_j$. 
Therefore, inductively, $Aut(G)$ preserves all the sets $O_i$, $i \in \{1, \ldots, r\}$.

The subgraph $G$ spanned by $O_1$ is a $|O_1|$-cycle. 
Hence, the group induced on $O_1$ by $Aut(G)$ is equal either to $C_{|O_1|}$ or to $D_{|O_1|}$. 
We have to show that the former holds. 
Assume that $\s \in Aut(G)$ fixes $v_0^1$. 
We show that $\s$ is the trivial permutation. 
If $\s$ does not acts trivially on $O_1$, then $\s(v_1^1) = v_{p^n-1}^1$ and $\s(v_{p^n-1}^1)=v_1^1$. 
Moreover, this implies that $\s(v_2^1) = v_{p^n-2}^1$ and $\s(v_{p^n-2}^1)=v_2^1$. 
The vertex $v_0^2$ is the only common $1$-neighbor of $v_0^1$ and $v_{p^n-1}^1$. 
The vertex $v_1^2$ is the only common $1$-neighbor of $v_0^1$ and $v_1^1$. 
Therefore, $\s(v_0^2) = v_1^2$ and $\s(v_1^2)=v_0^2$. 
However, $v_1^2$ is a $1$-neighbor of $v_{p^n-2}^1$ but, since $p > 5$, $v_0^2$ is not a $1$-neighbor of $v_2^1$, a contradiction. 
Thus, $\s$ acts trivially on $O_1$. 
Now, observe that $v_i^2$ is the only common $1$-neighbor of $v_i^1$ and $v_{(i-1 \pmod {p^n})}$. 
Hence, $\s$ acts trivially also on the set $O_2$. 
Moreover, for $k \geq 2$ a vertex $v_i^{k+1}$ is the only $1$-neighbor of $v_i^k$ that belongs to $O_{k+1}$. 
Therefore, inductively, $\s$ act trivially on each set $O_k$.  
Thus, $Aut(G) = A$. 
\hfill \cnd

Observe that the only two places, where we use the fact that $p>5$ are the facts that the cardinality of $O_2$ is greater than five, and that that the $1$-degrees of the vertices that do not belong to $O_1$ is different from $5$. 
The second property, we may force also by alternatively coloring using color $1$ all the edges that are contained in some $O_i$. 
Hence, this proof works also under the assumption that there are at least two orbits with cardinality at least $7$ instead of $p>5$. 
Therefore, it remains to consider the case when there is exactly one orbit of the cardinality greater than $5$. 
The only exception is when all the orbits are of cardinality one, two or four. This case is considered later.

First, we exclude the case $p=2$, where there is no orbits of cardinality $2$. 
In this situation, we have exactly one orbit of the maximal cardinality $p^n$. 
The rest of the nontrivial orbits are of cardinality three, four or five.
We have the following.

\begin{Lemma}\label{potega3}
In the situation as above, if the permutation group $A$ has exactly two nontrivial orbits, then $A \in GR^*(3)$. 
\end{Lemma}

\begin{proof}
Using Theorem~\ref{suma2} and Lemma~\ref{trywialne}, we may assume that $(A,V)$ has no nontrivial orbit. 
Hence, $V = O_1 \cup O_2$, where $O_1 = \{v_0, \ldots, v_{p^n-1}\}$, $O_2 = \{w_0, \ldots, w_{p-1}\}$ for $p \in \{3,5\}$ and $O_2 = \{w_0,w_1,w_2,w_3\}$ for $p=2$.

First we show that $A \notin GR(2)$. 
The proof is similar to the proof of Lemma~\ref{nie3}.
We consider an action of $A$ on the set of the edges. 
We have three types of orbits under this action. 
An orbit of the first type consists of the edges of the form $\{v_i,v_j\}$. 
An orbit of the second type consists of the edges of the form $\{w_i,w_j\}$. 
It is obvious that any coloring of these two types of orbits is preserved by the group $D_{p^n} \oplus D_p$ for $p \in \{3,5\}$ and by the group $D_{2^n} \oplus D_4$ for $p=2$.

Now, we may color the orbits that consists the edges of the form $\{v_i,w_j\}$. 
As in the proof of Lemma~\ref{nie3}, we have non-similar colorings of these orbits. 
For $p=3$, there are two non-similar colorings. For $p=2$ and for $p=5$, there are four of them. 
The colorings are exactly of the same kind as in the proof of the Lemma~\ref{nie3}.
As before, every such a coloring is preserved by some permutation that belongs to $(D_{p^n} \oplus D_p) \setminus A$ for $p \in \{3,5\}$ and to $(D_{2^n} \oplus D_4) \setminus A$ for $p=2$. 
Thus, $A \notin GR(2)$. 

Now, we prove that $A \in GR(3)$. 
We define a $3$-colored graph $G=(V,E)$ such that $Aut(G) = A$. 
We have to define an edge function $E$. Let us denote $n_p=p$ for $p \in \{3,5\}$ and $n_p=4$ for $p=2$. Then
$$
E(v,w) = \left\{
\begin{array}{cl}
1& \textrm{for } \{v,w\} = \{v_i, v_{i+1}\}, \\
1& \textrm{for } \{v,w\} = \{w_i, w_{i+1}\},  \\
1& \textrm{for } \{v,w\} = \{v_i, w_{i}\},  \\
2& \textrm{for } \{v,w\} = \{v_i, w_{i+1}\},  \\
0& \textrm{otherwise.}
\end{array}
\right.
$$
(The addition in the indices is modulo $p^n$ in the first case, and modulo $n_p$ in the remaining cases).

It is clear that $Aut(G) \supseteq A$. 
Since $A$ acts regularly on $O_1$, it is enough to show that $Aut(G)$ preserves $O_1$ and the stabilizer of $v_0$ is trivial. 
A vertex of the form $v_i$ has exactly one $2$-neighbor. 
A vertex of the form $v_j$ has exactly $\frac{p^n}{n_p}$ $2$-neighbors. 
Since $p^n>n_p$ the partition on to orbits $O_1$ and $O_2$ is preserved by $Aut(G)$.

Let $\s \in Aut(G)$ that fixes $v_0$. 
The vertex $w_0$ is the only $1$-neighbor of $v_0$ that belongs to $O_2$. 
Thus, $\s$ fixes $w_0$. 
The vertices $v_1$ and $v_{p^n-1}$ are the only $1$-neighbors of $v_0$ that belong to $O_1$. 
Moreover, $v_{p^n-1}$ is $2$-neighbor of $w_0$ and $v_1$ is not. 
Hence, $\s$ fixes $v_1$ and $v_{p^n-1}$. 
We use an induction to show that $\s$ fixes all the vertices that belong to $O_1$. 
Assume that $\s$ fixes the vertices $v_0, v_1, \ldots, v_i$ for $i \geq 1$. 
The vertex $v_i$ has exactly two $1$-neighbors that belong to $O_1$, vertices $v_{i-1}$ and $v_{i+1}$. 
Since, by assumption $\s$ fixes $v_{i-1}$, it fixes also $v_{i+1}$. This completes the step of induction. 
Hence, $\s$ fixes every vertex that belongs to $O_1$.

Finally, a vertex $w_i$ is the only $1$-neighbor of the vertex $v_i$ that belongs to $O_2$. 
Hence, $\s$ fixes every vertex that belongs to $O_2$, and therefore, $\s$ is trivial. 
This completes the proof of the lemma.  
\end{proof}

We still assume that there is exactly one orbit of cardinality greater than five and there is no orbit of cardinality two. 

\begin{Lemma}\label{conajmniej-trzy}
With the assumptions above, if $A$ has at least three nontrivial orbits, then $A \in GR(2)$. 
\end{Lemma}

\begin{proof}
The proof is similar to the proof of Lemma~\ref{p2}.
Let $n_p = p$ for $p \in \{3,5\}$ and $n_2=4$. 
By Theorem~\ref{suma2}, we may assume that there is no nontrivial orbit. 
Let $r \geq 3$ be the number of nontrivial orbits. 
We construct a graph $G=(V,E)$.
Let $V = \bigcup_{i=1}^r O_i$, where the sets $O_1 = \{v_0^1, \ldots v_{p^n}^1-1\}$, and $O_i = \{ v_0^i, \ldots, v_{n_p-1}^i\}$, $i \in \{2, \ldots r\}$ are the orbits of the group $A$.
$$E(\{v,w\})=\left\{
\begin{array}{cl}
1 & \textrm{for } \{v,w\} = \{v_i^1, v_{i+j}^1\}, \textrm{ where } j \in \{2, \ldots p^n-2\},\\ 
1 & \textrm{for } \{v,w\} = \{v_i^j, v_i^{l}\}, \\
&  \textrm{where either } l = j+1 \textrm{ or } j=1, l=3,\\ 
1 &  \textrm{for } \{v,w\} = v_i^1, v_{i+1}^2, \\
0 & \textrm{otherwise}.
\end{array}\right.$$ 
(The addition in the  indices is modulo $r^t$ in the first case, and modulo $n_p$ in the remaining cases).

It is clear that $Aut(G) \supseteq A$. 
We have to prove the opposite inclusion. 
Since $A$ acts regularly on $O_1$, it is enough to show that $A$ preserves the orbit $O_1$ and the stabilizer of $v_0^1$ is trivial. 

Every member of $O_1$ has exactly $p^n-1$ $1$-neighbors. 
Every member of $O_2$ has exactly $2\frac{p^n}{n_p}+1$ $1$-neighbors. 
Every member of $O_3$ has $\frac{p^n}{n_p}+1$ $1$-neighbors, if $t=3$ and $\frac{p^n}{n_p}+2$ $1$-neighbors, otherwise. 
The remaining vertices have at most two $1$-neighbors. 
Since $n_p \geq 3$, we have $p^n-1 > 2\frac{p^n}{n_p}+1>\frac{p^n}{n_p}+2>2$. 
Thus, the orbits $O_1$, $O_2$, and $O_3$ are preserved by $Aut(G)$.

Let $\s \in Aut(G)$ fixes $v_0^1$. 
The vertex $v_0^3$ is the only $1$-neighbor of $v_0^1$ that belongs to $O_3$. 
Hence, $\s$ fixes $v_0^3$. 
The vertex  $v_0^2$ is the only $1$-neighbor of $v_0^3$ that belongs to $O_2$.
Hence, $\s$ fixes $v_0^2$.
The vertices $v_1^1$ and $v_{p^n-1}^1$ are the only two $0$-neighbors of $v_0^1$ that belong to $O_1$. 
Moreover, $v_{p^n-1}^1$ is $1$-neighbor of $v_0^2$ and $v_1^1$ is not.
Hence, $\s$ fixes $v_1^1$ and $v_{p^n-1}^1$. 
The same induction argument as in the proof of Lemma~\ref{potega3} shows that $\s$ fixes every member of $O_1$.

The vertex $v_i^3$ is the only $1$-neighbor of $v_i^1$ that belongs to $O_3$. 
Thus, $\s$ fixes all the members of $O_3$. 
Similarly, $v_i^2$ is the only $1$-neighbor of $v_i^3$ that belongs to $O_2$. 
Hence, $\s$ fixes all the members of $O_2$.

We prove by induction that $\s$ fixes all the vertices. 
Assume that $\s$ fixes all the elements of the set $O_1 \cup \cdots \cup O_i$ for $i \geq 3$. 
The vertex $v_j^{i+1}$ is the only $1$-neighbor of $v_j^i$ that does not belongs to the set $O_1 \cup \cdots \cup O_i$. 
Thus, $\s$ fixes every member of $O_{i+1}$. This finishes the induction step. 
Consequently, $\s$ fixes all the vertices, and therefore, $\s$ is trivial.
Hence $Aut(G) \subseteq A$. This completes the proof of the lemma. 
\end{proof}

Lemma~\ref{conajmniej-trzy} decides the last case for the primes greater than two. 
In this situation, we have the following result. 

\begin{Theorem} \label{nie2}
Let $A$ be a cyclic permutation group of order $p^n$. Let $k_i$, $i \in \{1, \ldots, n\}$ denote the number of orbits of $A$ of cardinality $p^i$.
If $p \ne 2$, then 
\begin{enumerate}
    \item if $\sum_{i=1}^n k_i= 1$, then $A \notin GR$, 
    \item if $\sum_{i=1}^n k_i= 2$, then 
          \begin{itemize}
    \item $A \in GR^*(3)$, for $k_1 \in \{ 1, 2 \}$ and $p \in \{ 3, 5 \}$, 
\item $A \in GR(2)$, otherwise,
\end{itemize}
\item if $\sum_{i=1}^n k_i> 2$, then $A \in GR(2)$.
\end{enumerate}
\end{Theorem}

A little different situation is when $p=2$. We still do not consider here the situation, when there are orbits of cardinality two. 
The case that we still do not consider for $p=2$ is when we have at least one orbit of cardinality two and at least one orbit of cardinality greater than two and at most one orbit of cardinality greater than four. 
First, we consider the following.  

\begin{Lemma}
If $A$ has at least one orbit of cardinality two and exactly one orbit of greater cardinality, then $A \notin GR$. 
\end{Lemma} 

\begin{proof}
By Lemma~\ref{trywialne}, we may assume that there is no nontrivial orbit. 
Let $O_i = \{v_0^i, \ldots, v_{n_i}^i\}$, $i \in \{1, \ldots, k\}$ be the orbits of $A$. 
Assume that $n_1 > 2$, then $n_i = 2$ for $i >1$. 
We consider the action of $A$ on the set of the edges. 
As usually, it is no matter how we color the orbits that consists of the edges that are contained in $O_1$, such a coloring will be preserved by the permutation 
$$\s = \Pi (v_{j}^i,v_{(n_i - j-1)}^i), \textrm{ where } j \in \{0, \ldots \frac{n_i-1}{2}\}, i \in \{1, \ldots, k\}. $$
It is also clear that any of the remaining orbits is preserved by $\s$. 
Obviously, $\s \notin A$. Hence, $A \notin GR$. 
\end{proof}

The remaining case is when we have at least one orbit of cardinality two and at least two orbits of greater cardinality with at most one of them of cardinality greater than four.

\begin{Lemma}
With the assumption above, $A \in GR(2)$. 
\end{Lemma}

\begin{proof}
By Theorem~\ref{suma2}, we may assume that there is no nontrivial orbit. 
Assume that $A$ has $r\geq 1$ orbits of cardinality two and $t \geq 2$ orbits of cardinality greater than two. 
Let $O_1 = \{v_0^1, \ldots, v_{p^n-1}^1\}$, $O_j = \{v_0^j,v_1^j,v_2^j,v_3^j\}$, $j \in \{2, \ldots, t\}$, and $P_i=\{w_0^i,w_1^i\}$, $i \in \{1, \ldots, r\}$ be the orbits of $A$.  
We define a $2$-edge-colored graph $G=(V,E)$ such that $Aut(G) = A$. 
In Figure~\ref{c442} the simplest graph in the case with three orbits and $|O_1|=4$ is presented. 
In general case, we generalize and modify this construction to obtain a common proof for all cases.

\begin{figure}%%%%%%%%%%%%%%%%%%%%%%%%%%%%%%%%%%%%%%%%%%%%%%%%%%%%%%%%%%%%%%%%%%%%%
    \psset{linewidth=0.03, xunit=1.7cm, yunit=1.5cm}
\begin{pspicture}(4,4)(0,0)

\cnode(0,0){0.1}{b1}
\cnode(4,0){0.1}{b2}
\cnode(4,4){0.1}{b3}
\cnode(0,4){0.1}{b4}
\cnode(2,0){0.1}{c1}
\cnode(4,2){0.1}{c2}
\cnode(2,4){0.1}{c3}
\cnode(0,2){0.1}{c4}

\cnode(1,2){0.1}{d1}
\cnode(2,3){0.1}{d2}

 \psset{linewidth=0.05}

\ncline[linecolor=black]{-}{c1}{c2}
\ncline[linecolor=black]{-}{c2}{c3}
\ncline[linecolor=black]{-}{c3}{c4}
\ncline[linecolor=black]{-}{c4}{c1}

\ncline[linecolor=black]{-}{b1}{c1}
\ncline[linecolor=black]{-}{b2}{c2}
\ncline[linecolor=black]{-}{b3}{c3}
\ncline[linecolor=black]{-}{b4}{c4}

\ncline[linecolor=black]{-}{c1}{b2}
\ncline[linecolor=black]{-}{c2}{b3}
\ncline[linecolor=black]{-}{c3}{b4}

\ncline[linecolor=black]{-}{c4}{b1}
\ncline[linecolor=black]{-}{d1}{b1}
\ncline[linecolor=black]{-}{d1}{c1}

\ncline[linecolor=black]{-}{d1}{b3}
\ncline[linecolor=black]{-}{d1}{c3}

\ncline[linecolor=black]{-}{d2}{b2}
\ncline[linecolor=black]{-}{c2}{d2}

\ncline[linecolor=black]{-}{d2}{c4}
\ncline[linecolor=black]{-}{d2}{b4}

\end{pspicture}
\caption{A $2$-edge-colored graph $G$ such that $Aut(G)$ is cyclic permutation group with three orbits of cardinality four, four, and two.}\label{c442}
\end{figure}
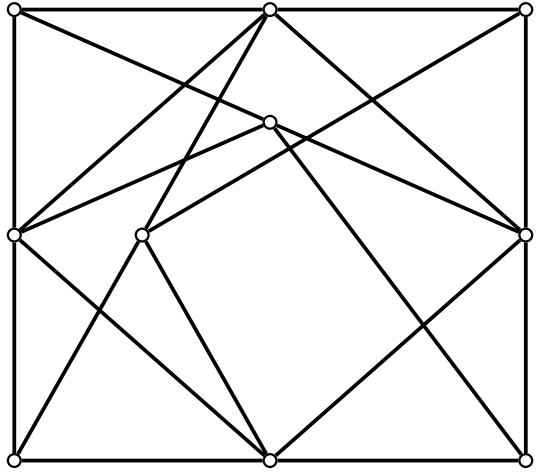

Let $V = \bigcup O_i$. 
Then, 
$$E(\{v,w\}) = 
\left\{
\begin{array}{cl}
1 & \textrm{for } \{v,w\} = \{v_j^1, v_{j+1}^1\},  \\
1 & \textrm{for } \{v,w\} \subset O_2, \\
1 & \textrm{for } \{v,w\} = \{v_j^i,v_{j}^{i+1}\}, i<t, \\
1 & \textrm{for } \{v,w\} = \{v_j^1,v_{j+1}^2\},\\
1 & \textrm{for } \{v,w\} = \{v_j^i,w_{j}^1\}, i \in \{1,2\},\\
1 & \textrm{for } \{v,w\} = \{v_j^t,w_{j}^2\}, \textrm{ if } r \geq 2,\\
1 & \textrm{for } \{v,w\} = \{w_j^i,w_j^{i+1}\}, 1<i<t,\\
0 & \textrm{otherwise.}
\end{array}
\right.$$
(The addition in the  indices is modulo $p^n$ in the first case, and modulo $4$ or $2$, respectively, in the remaining cases).

It is clear that $A \subseteq Aut(G)$. 
We have to prove the opposite inclusion. 
Since $A$ acts regularly on $O_1$, it is enough to show that $Aut(G)$ preserves $O_1$ and the stabilizer of $v_0^1$ is trivial.

The $1$-degree of a member of $O_1$ is exactly five. 
The $1$-degree of a member of $O_2$ is at least $\frac{p^n}{2}+4>5$.  
The $1$-degree of a member of $P_1$ is  $\frac{p^n}{2}+2 \ne 5$. 
The rest of the vertices has $1$-degree at most two. 
This shows that $Aut(G)$ preserves the orbits $O_1$, $O_2$, and $P_1$.

Let $\s \in Aut(G)$ fixes $v_0^1$. 
The vertex $w_0^1$ is the only $1$-neighbor of $v_0^1$ that belongs to $P_1$. 
Thus, $\s$ fixes $w_0^1$ and $w_1^1$.  
The vertex $v_0^2$ is the only common $1$-neighbor of $v_0^1$ and $w_0^1$. 
Hence, $\s$ fixes $v_0^2$. 
Similarly, the vertex $v_1^2$ is the only common $1$-neighbor of $v_0^1$ and $w_1^1$. 
Hence, $\s$ fixes $v_1^2$.
The vertex $w_0^1$ has exactly two $1$-neighbors that belong to $O_2$: $v_0^2$ and $v_2^2$. 
Therefore, $\s$ fixes $v_2^2$, and consequently, $\s$ fixes $v_3^2$. 
Continuing, $v_1^1$ is the only common $1$-neighbor of $v_1^2$ and $v_2^4$ among the members of $O_1$. 
Thus, $\s$ fixes $v_1^1$. 
The same induction argument as many times before shows that $\s$ fixes every member of $O_1$. 

For $i \geq 2$, the vertex $v_j^{i+1}$ is the only $1$-neighbor of the vertex $v_j^i$ that is outside the set $P_1 \cup \bigcup_{l=1}^i O_l$. 
Therefore, inductively, $\s$ fixes every vertex in all orbits $O_i$. 
If $r \geq 2$, then the vertex $w_j^2$ is the only $1$-neighbor of the vertex $v_j^t$ that is outside the set $P_1 \cup \bigcup O_i$. 
Thus, $\s$ fixes $w_0^2$ and $w_1^2$. 
Finally, for $i \geq 2$ the vertex $w_j^{i+1}$ is the only $1$-neighbor of the vertex $w_j^i$ that is outside the set $P_{i-1} \cup O_t$. 
Therefore, inductively, $\s$ fixes every vertex in all orbits $P_i$. 
This implies that $\s$ is trivial. 
Therefore, $Aut(G) \subseteq A$. This completes the proof of the lemma.
\end{proof}

Summarizing all the cases for $p=2$, we have the following.

\begin{Theorem} \label{2}
Let $A$ be a cyclic permutation group of order $2^n$. Let $k_i$, $i \in \{1, \ldots, n\}$ denote the number of orbits of $A$ of cardinality $2^i$. 
Then, 
\begin{enumerate}
    \item if $\sum_{i=2}^n k_i= 1$, then $A \notin GR$, 
     \item if $\sum_{i=2}^n k_i= 2$, $k_1=0$ and $k_2 \in \{1,2\}$, then $A \in GR^*(3)$, 
     \item $A \in GR(2)$, otherwise.
\end{enumerate}
\end{Theorem}

Now, for the cyclic automorphism groups of colored graphs of prime power order we have the following simple characterization. 

\begin{Corollary}
A cyclic permutation group $A$ of prime power order belongs to $GR$ if and only if $A$ has at least two orbits of cardinality greater than two or $A$ has order two.
\end{Corollary}

\end{document}